\documentclass[11pt]{amsart}
\usepackage{amsaddr}
\usepackage{amscd}
\usepackage{amsmath}
\usepackage{amsthm}
\usepackage{cases}
\usepackage{chngcntr}
\usepackage{cite}
\usepackage{color}
\usepackage{enumitem}
\usepackage{graphicx}
\usepackage{mathtools}
\usepackage{microtype}
\usepackage[active]{srcltx}
\usepackage{thmtools} % Prevents cleveref from calling everything a theorem
\usepackage{tikz}
\usepackage{url}
\usepackage{verbatim} % use \begin{verbatim}...\end{verbatim}

\usepackage[UKenglish]{babel}
\usepackage[UKenglish]{datetime}

\usepackage{hyperref}
\usepackage{cleveref}

\usepackage{amsfonts}
\usepackage{bbm} % extra blackboard symbols, serif and sans serifto
\usepackage[charter]{mathdesign}
\usepackage[mathscr]{euscript}

\usepackage{calrsfs}
\DeclareMathAlphabet{\mathcalalt}{OMS}{cmsy}{m}{n}

\usepackage{ifthen}

\makeatletter
\newcommand{\tfs}[1]
{
	\ifthenelse{\equal{\f@shape}{n}}{\ensuremath{\mathrm{#1}}}
	{\ifthenelse{\equal{\f@shape}{sc}}{\ensuremath{\mathrm{#1}}}
		{\ifthenelse{\equal{\f@shape}{it}}{\ensuremath{\mathit{#1}}}
			{\ifthenelse{\equal{\f@shape}{sl}}{\ensuremath{\mathit{#1}}}{}	
			}
		}
	}
}
\makeatother

\makeatletter
\newcommand{\btfs}[1]
{
	\ifthenelse{\equal{\f@shape}{n}}{\ensuremath{\mathrm{#1}}}
	{\ifthenelse{\equal{\f@shape}{sc}}{\ensuremath{\mathrm{#1}}}
		{\ifthenelse{\equal{\f@shape}{it}}{\ensuremath{\mathit{#1}}}
			{\ifthenelse{\equal{\f@shape}{sl}}{\ensuremath{\mathit{#1}}}{}	
			}
		}
	}
}
\makeatother

\theoremstyle{plain}

\newtheorem{theorem0}{Theorem}[section]
\newtheorem{theorem}[theorem0]{Theorem}
\newtheorem{proposition}[theorem0]{Proposition}
\newtheorem{lemma}[theorem0]{Lemma}
\newtheorem{corollary}[theorem0]{Corollary}

\newtheorem*{theorem*}{Theorem}
\newtheorem*{proposition*}{Proposition}
\newtheorem*{lemma*}{Lemma}
\newtheorem*{corollary*}{Corollary}

\theoremstyle{definition}

\newtheorem{definition}[theorem0]{Definition}

\newtheorem{remark}[theorem0]{Remark}

\newtheorem*{definition*}{Definition}
\newtheorem*{example*}{Example}
\newtheorem*{remark*}{Remark}

\setlist[enumerate,1]{label=\textup{(\arabic*)},ref=\arabic*}
\setlist[enumerate,2]{label=\textup{(\alph*)},ref=\arabic{enumi}.\alph*}
\setlist[enumerate,3]{label=\textup{(\roman*)},ref=\arabic{enumi}.\alph{enumii}.\roman*}
\setlist[enumerate,4]{label=\textup{(\Alph*)},ref=\arabic{enumi}.\alph{enumii}.\roman{enumiii}.\Alph*}

\crefname{theorem}{Theorem}{Theorems}
\crefname{proposition}{Proposition}{Propositions}
\crefname{lemma}{Lemma}{Lemmas}
\crefname{corollary}{Corollary}{Corollaries}
\crefname{conjecture}{Conjecture}{Conjectures}
\crefname{definition}{Definition}{Definitions}
\crefname{example}{Example}{Examples}
\crefname{remark}{Remark}{Remarks}
\crefname{equation}{equation}{equations}

\crefname{section}{Section}{Sections}
\crefname{subsection}{Section}{Sections}
\crefname{subsubsection}{Section}{Sections}

\crefname{enumi}{part}{parts}
\crefname{enumii}{part}{parts}
\crefname{enumiii}{part}{parts}
\crefname{enumiv}{part}{parts}

\newcommand{\enclosepart}[1]{(#1)}
\newcommand{\partref}[1]{\enclosepart{\ref{#1}}}

\numberwithin{equation}{section}
\allowdisplaybreaks % displayed equations can be broken across pages

\newcommand{\tpskip}{\hskip  0.05555555555em}
\newcommand{\tnskip}{\hskip -0.05555555555em}

\newcommand{\bbfont}{\mathbb}

\newcommand{\CC}{{\bbfont C}}

\newcommand{\RR}{{\bbfont R}}
\newcommand{\TT}{{\bbfont T}}

\newcommand{\ulp}{{\textup{(}}}
\newcommand{\urp}{{\textup{)}}}
\newcommand{\uppars}[1]{\ulp #1\urp}

\newcommand{\abs}[1]{{\lvert #1 \rvert}}
\newcommand{\norm}[1]{{\lVert #1 \rVert}}
\newcommand{\braces}[1]{{\{ #1\}}}

\newcommand{\lrabs}[1]{{\left\lvert #1 \right\rvert}}

\newcommand{\set}[1]{\braces{\,#1\,}}

\newcommand{\Dc}{De\-de\-kind com\-plete}
\newcommand{\sDc}{$\sigma$-De\-de\-kind com\-plete}

\newcommand{\noarg}{\,\cdot\,}

\newcommand{\up}{{\Uparrow}}
\newcommand{\down}{{\Downarrow}}
\newcommand{\ups}{{\uparrow}}
\newcommand{\downs}{{\downarrow}}

\newcommand{\supp}{\mathrm{supp}}

\newcommand{\sosum}{\sigma\tnskip\tfs{o}\,\text{\textendash}\!\sum}

\newcommand{\oto}{\overset{\tfs{o}}{\to}}
\newcommand{\soto}{\overset{\sigma\tfs{o}}{\to}}

\newcommand{\Calgebra}{${\tfs{C}}^\ast$\!-algebra}
\newcommand{\Calgebras}{${\tfs{C}}^\ast$\!-algebras}
\newcommand{\sigalgebra}{$\sigma$-algebra}

\newcommand{\specop}{\sigma(\op)}
\newcommand{\borelspecop}{\borel(\specop)}

\newcommand{\pset}{X}

\newcommand{\ts}{X}

\newcommand{\proj}[1]{P_{#1}}

\newcommand{\iu}{{\tfs{i}\mkern1mu}}

\newcommand{\nb}{\tfs{nb}}
\newcommand{\reg}{\tfs{r}}

\newcommand{\zerofunction}{\mathbf{0} }
\newcommand{\onefunction}{\mathbf{1}}
\newcommand{\indicator}[1]{\chi_{#1}}

\newcommand{\pos}[1]{{#1^+}}

\newcommand{\seq}[1]{({#1}_n)_{n=1}^{\infty}}
\newcommand{\net}[1]{({#1}_\lambda)_{\lambda\in \Lambda}}

\newcommand{\widehatss}{{\widehat{\phantom{n}}}}

\newcommand{\ounorm}[1]{\lVert #1 \rVert_{\idmap}}

\newcommand{\idfunction}{\tfs{id}}
\newcommand{\idmap}{I}
\newcommand{\mult}{M}
\newcommand{\op}{T}
\newcommand{\posmap}{\pi}
\newcommand{\optwo}{S}

\newcommand{\vl}{E}
\newcommand{\bl}{E}
\newcommand{\vltwo}{F}

\newcommand{\vlC}{\vl_\CC}
\newcommand{\blC}{\bl_\CC}
\newcommand{\vltwoC}{\vltwo_\CC}

\newcommand{\posvl}{\pos{\vl}}
\newcommand{\posbl}{\pos{\bl}}
\newcommand{\posR}{\pos{\RR}}
\newcommand{\posRext}{\overline{\pos{\RR}}}

\newcommand{\cont}{\tfs{C}}
\newcommand{\conto}{\cont_0}
\newcommand{\contc}{\cont_{\tfs{c}}}
\newcommand{\conttsR}{\cont(\ts,\RR)}

\newcommand{\contctsR}{\contc(\ts,\RR)}
\newcommand{\conttsC}{\cont(\ts,\CC)}
\newcommand{\contotsC}{\conto(\ts,\CC)}

\newcommand{\contspectrumR}{\cont(\specop,\RR)}
\newcommand{\contspectrumC}{\cont(\specop,\CC)}

\newcommand{\linear}{{\tfs L}}
\newcommand{\nbounded}{\linear_{\nb}}
\newcommand{\regular}{\linear_{\reg}}

\newcommand{\nboundedblC}{\nbounded(\blC)}

\newcommand{\regularblC}{\regular(\blC)}

\newcommand{\centre}{\mathcalalt{Z}}

\newcommand{\centrebl}{\centre(\bl)}
\newcommand{\centreblC}{\centre(\blC)}

\newcommand{\mss}{\Delta}
\newcommand{\alg}{\Omega}
\newcommand{\borel}{\mathscr B}

\newcommand{\npm}{\mu}
\newcommand{\imnpm}{\npm_{\op}}

\newcommand{\di}[1]{\,{\tfs d} #1}

\newcommand{\orderintegralgeneral}[3]{{\int_{#1}^{\tfs {o}}\!\!\! {#2}\di {#3}}}
\newcommand{\orderintegralspectrum}[3]{{\int_{#1}^{\tfs {o}}\!\!\!\!\!\!\!\! {#2}\di {#3}}}
\newcommand{\ointnpmgeneral}[1]{\orderintegralgeneral{\pset}{#1}{\npm}}

\newcommand{\ointts}[1]{\orderintegralgeneral{\ts}{#1}{\npm}}
\newcommand{\ointspectrum}[1]{\orderintegralspectrum{\sigma(\op)}{#1}{\imnpm}}
\newcommand{\ointspectrumalt}[1]{\orderintegralspectrum{\sigma(\op)}{#1}{\tnskip\tnskip\nu}}

\newcommand{\boundedmeasfunspectrumR}  {{\mathcal B}(\specop,\RR)}
\newcommand{\boundedmeasfunspectrumC}  {{\mathcal B}(\specop,\CC)}

\usepackage[mathlines]{lineno}

\begin{document}

%%%%%%%%%%%%%%%%%%%%%%%%%%%%%%%%% BEGIN FRONTMATTER %%%%%%%%%%%%%%%%%%%%%%%%%%%%%%%%%%%%
%\title[The ideal centre of a complex Banach lattices]{The ideal centre of a complex Banach lattice}
\title[Central operators on complex Banach lattices]{Central operators on complex Banach lattices}

\author{Marcel de Jeu}
\address[Marcel de Jeu]{Mathematical Institute, Leiden University, P.O.\ Box 9512, 2300 RA Leiden, The Netherlands\\
and\\
Department of Mathematics and Applied Mathematics, University of Pretoria, Corner of Lynnwood Road and Roper Street, Hatfield 0083, Pretoria,
South Africa}
\email[Marcel de Jeu]{mdejeu@math.leidenuniv.nl}

\author{Xingni Jiang}
\address[Xingni Jiang]{College of Mathematics, Sichuan University, No.\ 24, South Section, First Ring Road, Chengdu, P.R.\ China}
\email[Xingni Jiang]{x.jiang@scu.edu.cn}

\begin{abstract}We show that the centre of a Dedekind complete complex Banach lattice is a commutative $\mathrm{C}^\ast$-algebra in the order unit norm. This implies that the order unit norm and the operator norm coincide. As an application of the latter, a Fuglede--Putnam--Rosenblum-type theorem is established. Under an extra condition on the underlying real Banach lattice, which is satisfied when it is order continuous, a spectral theorem is given for the centre of the complex Banach lattice as a whole and for an individual central operator. The ensuing functional calculus for an individual operator is applied to show that central operators have many spectral properties similar to those of normal operators on complex Hilbert spaces. For example, when the spectrum is countable every element is the sum of an order convergent series of eigenvectors.
\end{abstract}

%\makeatletter{\renewcommand*{\@makefnmark}{}
%	\date{}\footnote{{File: \currfilebase. Compiled: \today,\,\currenttime.}}
%	\makeatother}

\maketitle

\section{Introduction and overview}\label{sec:introduction_and_overview}

\noindent Central operators on real Banach lattices form a well-studied class operators. This is much less so for central operators on complex Banach lattices, the study of which we take up in this paper with complex \Calgebras\ central to our approach. The latter is in contrast to the current literature on the real case, where it is exploited that in the order unit norm the central operators are isometrically isomorphic to $\conttsR$ as a Banach lattice algebra for a compact Hausdorff space $\ts$, but where it is hardly ever used that $\conttsR$ is even a real \Calgebra.

This paper is organised as follows.

In \cref{sec:preliminaries}, we include the necessary notation and definitions and a little material about complexifications. We give a concise definition of the so-called order integral. This generalisation of the Lebesgue integral is the language for the representation theorems in \cite{de_jeu_jiang:2021c} and \cite{de_jeu_jiang:2022b} that are at the basis of most of the spectral theory in the current paper; it is also the means to establish isomorphisms between vector and Banach lattices of operators and vector and Banach lattice of measures in \cite{de_jeu_jiang:2025}.\footnote{The seminal work for this type of integrals is by Wright; see \cite{wright:1969a,wright:1969b}, among others. Independently of the present authors, Kusraev and Tasoev developed the theory of the closely related so-called Kantorovich-Wright integral; see\cite{kusraev_tasoev:2017}.} Two (folklore) facts about complex \Calgebras\ are also included in this section.

In \cref{sec:central_operators}, various aspects of the central operators on complex Banach lattices are studied from a complex \Calgebra\ point of view.  This makes it easy to show that the order unit norm and the operator norm on the centre are equal, which can then be used to establish a Fuglede-Putnam-Rosenblum--type theorem.

The spectral theory for special classes of operators on complex Banach lattices is rather limited. There appears to be little, if anything at all, in addition to the work in \cite{wickstead:1982} for compact lattice homomorphisms on complex Banach lattices. In \cref{sec:spectral_theory}, we develop such a theory for the central operators on complex Banach lattices. A representation theorem from \cite{de_jeu_jiang:2021c} that has its roots in \cite{de_jeu_jiang:2022b} gives a spectral measure for the whole centre, which is then used to obtain one for an individual central operator. The resulting functional calculus is employed to show that central operators share many properties with normal operators on complex Hilbert spaces, including the development of an arbitrary element as an unconditionally $\sigma$-order convergent series of eigenvectors when the spectrum is countable.

\section{Preliminaries}\label{sec:preliminaries}

\noindent All real vector lattices in this paper are supposed to be Archimedean. All \Calgebras\ are supposed to be complex. All operators are supposed to be linear; the identity operator on a vector space is denoted by $\idmap$. The characteristic function of a subset $S$ of a set $\pset$ is denoted by $\chi_S$. By $\zerofunction$, resp.\ $\onefunction$, we mean a function that is zero, resp.\ 1, everywhere. We use self-evident notations for spaces of continuous functions. The Borel \sigalgebra\ of a topological space $\ts$ is denoted by $\borel(\ts)$.

\subsection{Complexification}\label{subsec:complexification}
We shall employ the usual basic definitions and notations regarding the complexifications of uniformly complete vector real lattices and of real Banach lattices. In particular, we let $\regularblC$,  denote the regular operators on a complex vector lattice. For details, we refer to \cite[Section~3.2]{abramovich_aliprantis_INVITATION_TO_OPERATOR_THEORY:2002},  \cite[Chapter~2, \S~11]{schaefer_BANACH_LATTICES_AND_POSITIVE_OPERATORS:1974}, \cite[Sections~91 and~92]{zaanen_RIESZ_SPACES_VOLUME_II:1983}, and \cite[Chapters~6 and 19]{zaanen_INTRODUCTION_TO_OPERATOR_THEORY_IN_RIESZ_SPACES:1997}.

For $z=x+\iu y$ in a complex vector lattice $\vlC$ with real part $x\in\vl$ and imaginary part $y\in\vl$, we set $\overline{z}\coloneqq x-\iu y$. Then $\abs{z}=\abs{\overline z}$.

We extend the notions of order convergence and of $\sigma$-order convergence to $\vlC$ by saying that $z_\alpha\oto z$ for a net $(z_\alpha)_{\alpha\in A}$ in $\vlC$ when there exists a net $(u_\beta)_{\beta\in B}$ in $\vl$ such that $u_\beta\downarrow 0$ in $\vl$, and with the property that, for every $\beta_0\in B$, there exists an $\alpha_0\in A$ such  that $\abs{z_\alpha-z}\leq u_\beta$ for all $\alpha\geq\alpha_0$. We say that $z_n\soto z$ for a sequence $(z_n)_{n=1}^\infty$ in $\vlC$ when there exists a sequence $(u_n)_{n=1}^\infty$ in $\vl$ such that $u_n\downarrow 0$ in $\vl$ and  $\abs{z_n-z}\leq u_n$ for all $n$. For a sequence $(z_n)_{n=1}^\infty$ in $\vlC$, we shall write  $z=\sosum_{n=1}^\infty z_n$ when $\sum_{n=1}^N z_n\soto z$ as $N\to\infty$.

When $\vltwoC$ is a complex vector lattice, an operator $\op\colon\vlC\to\vltwoC$ is said to be order continuous, resp.\ $\sigma$-order continuous, when it preserves order convergence of nets,  resp.\ $\sigma$-order convergence of sequences. This is the case if and only if the real and imaginary parts of $\op$ have the respective property.

Let $\ts$ be a compact Hausdorff space. The natural complex algebra isomorphism between $\conttsR+\iu\conttsR$ and $\conttsC$ is an isometric vector lattice isomorphism between $\conttsR_\CC$ and $\conttsC$.

\subsection{Order integrals}\label{subsec:order_integrals}

We briefly summarise the relevant facts about measures with values in the positive cone $\vl$ of a \sDc\ vector lattice $\vl$ and the integrals that they define. For some of the general representation theorems in \cite{de_jeu_jiang:2022b}, infinite measures with values in the extended positive cone of $\vl$ are needed but in the current paper finite ones are sufficient. We shall confine ourselves to this case.

\begin{definition}\label{def:positive_pososext_valued_measure}
	Let $\pset$ be a set, let $\alg$ be a \sigalgebra\ of subsets of $\ts$, and let $\vl$ be a \sDc\  vector real lattice. An \emph{$\vl$-valued measure} is a map $\npm:\alg\rightarrow \posvl$ such that the following hold:
	\begin{enumerate}
		\item $\npm(\emptyset)=0$;\label{part:pososext_valued_measure_1}
		\item if $\seq{\mss}$ is a pairwise disjoint sequence in $\alg$, then
		\begin{equation*}\label{eq:sigma_additivity}
			\npm\left(\bigcup_{n=1}^\infty\mss_n\right)=\sosum_{n=1}^\infty\npm(\mss_n).
		\end{equation*}		
	\end{enumerate}
	When $\vl$ is also a real associative algebra, $\npm$ is said to be a \emph{spectral measure} when $\npm(\mss_1\cap\mss_2)=\npm(\mss_1)\cdot\npm(\mss_2)$ for $\mss_1,\mss_2\in\alg$.\footnote{It is not required that $\npm(\pset)$ equals the identity element of $\vl$ when there is one.}
\end{definition}

An \emph{elementary} function is a measurable function $\varphi\colon\pset\to\posR$ that can be written as $\varphi=\sum_{i=1}^n r_i\indicator{\mss_i}$ with $\mss_i\in\alg$. The \emph{order integral} of $\varphi$ is then defined by setting $\ointnpmgeneral{\,\,\varphi}\coloneqq\sum_{i=1}^n r_i\npm(\mss_i)$; this does not depend on the choice of a decomposition $\varphi=\sum_{i=1}^n r_i\indicator{\mss_i}$. For an arbitrary measurable $f\colon\ts\to\posRext$, one chooses a sequence $\varphi_n$ of elementary functions such that $\varphi_n(x)\uparrow f(x)$ in $\posRext$ for every $x\in\ts$, and defines  $\ointnpmgeneral{\,\,f}\coloneqq\sup_{n\geq 1} \ointnpmgeneral{\,\,\varphi_n}$ as an element of the extended positive cone of $\vl$.  This is independent of the choice of the sequence. For a measurable $f\colon\pset\to\RR$ such that $\ointnpmgeneral{\,\,f^+}$ and $\ointnpmgeneral{\,\,f^-}$ are both finite, its order integral is defined by setting  $\ointnpmgeneral{\,\,f}\coloneqq\ointnpmgeneral{\,\,f^+}-\ointnpmgeneral{\,\,f^-}$. For a measurable complex-valued function $f$ such that $\ointnpmgeneral{\,\,\abs{f}}<\infty$, its order integral is defined in the obvious way. The further theory for this integral is developed in some detail in \cite{de_jeu_jiang:2022a}, including results such as the monotone convergence theorem, Fatou's lemma, and the dominated convergence theorem. The triangle inequality
\begin{equation*}\label{eq:triangle_inequality}
	\lrabs{\ointnpmgeneral{f}}\leq \ointnpmgeneral{\abs{f}}
\end{equation*}
holds for integrable $f\colon\pset\to\RR$; see \cite[Lemma~6.7]{de_jeu_jiang:2022a}. Since $\vl$, being \sDc\, is uniformly complete (see \cite[Theorem~12.8]{zaanen_INTRODUCTION_TO_OPERATOR_THEORY_IN_RIESZ_SPACES:1997}), its complexification $\vlC$ is defined. It is easily seen from the real case that the triangle inequality holds for integrable complex-valued $f$.

As for the case where $\vl=\RR$, for a compact Hausdorff space $\ts$ there are regularity properties to be considered for measures on $\borel(\ts)$.

\begin{definition}\label{def:regularity_of_measures}
	Let $\ts$ be a compact Hausdorff space, let $\vl$ be a \Dc\ real vector lattice, and let $\npm\colon\borel(\ts)\to\posvl$ be a measure. Then $\npm$ is a \emph{Borel measure}. It is said to be:
		\begin{enumerate}
		\item \emph{inner regular at $\mss\!\in\!\borel(\ts)$} if $\npm(\mss)\!=\!\sup\set{\!\npm(K)\!:\! K\ \text{is compact and}\ K\!\subseteq\! \mss\!}$;
		\item \emph{outer regular at $\mss\in\borel(\ts)$} if $\npm(\mss)=\inf\set{\!\npm(V): V\ \text{is open and}\ \mss\subseteq V\!}$;
		\item a \emph{regular Borel measure} if $\npm$ is inner regular at all open subsets of $\ts$ and outer regular at all Borel sets.\footnote{We follow the terminology as in \cite{aliprantis_burkinshaw_PRINCIPLES_OF_REAL_ANALYSIS_THIRD_EDITION:1998} for $\vl=\RR$. When following \cite{folland_REAL_ANALYSIS_SECOND_EDITION:1999}, our regular Borel measures would have been called Radon measures.}
	\end{enumerate}
\end{definition}

\subsection{Representation theorems}\label{subsec:representation_theorems} We shall now mention the relevant results from  \cite{de_jeu_jiang:2021c} and \cite{de_jeu_jiang:2022b}. For a locally compact Hausdorff space $\ts$, a positive linear map $\pi\colon \contctsR\to\vl$ with values in a suitable \Dc\ vector lattice $\vl$ can be represented by a finite Borel measure on $\ts$, provided that it maps norm bounded subsets of $\contctsR$ to order bounded subsets of $\vl$.\footnote{When $\vl$ is an order continuous Banach lattice, no  such restriction is necessary. The representing measure can then be infinite; see \cite[Theorem~4.2]{de_jeu_jiang:2022b}.} In our case, $\ts$ will be compact and then condition is automatically met.

We merge the specialisations of \cite[Theorem~6.4]{de_jeu_jiang:2021c} and \cite[Theorem~5.4]{de_jeu_jiang:2022b} to the case of compact $\ts$ in the following result, which is the basis for obtaining the spectral measures in \cref{sec:spectral_theory}.
We recall that a real vector lattice is called \emph{normal} when its order continuous dual separates its points. All order continuous real Banach lattices are normal.

\begin{theorem}\label{res:riesz_representation_theorems_merged}
	Let $\ts$ be a compact Hausdorff space, let $\vl$ be a \Dc\ and normal real vector lattice, and let $\posmap:\conttsR\rightarrow\vl$ be a positive linear map. Then there exists a unique regular Borel measure $\npm:\borel(\ts)\to\posvl$ on the Borel \sigalgebra\ of $\ts$ such that
	\[
	\posmap(f)=\ointnpmgeneral{f}
	\]
	for $f\in\conttsR$.
	
	When $V$ is a non-empty open subset of $\ts$,
	\begin{align}\label{eq:measure_of_open_subset}
		\npm(V)&=\sup\,\{\posmap(f) : f\in\conttsR,\, \zerofunction\leq f\leq\onefunction,\, \supp \,f\subseteq V\}.\\
	\intertext{When $K$ is a non-empty compact subset of $\ts$,}
	\label{eq:measure_of_compact_subset}
		\npm(K)&=\inf\,\{\posmap(f) : f\in\conttsR,\, \zerofunction\leq f\leq\onefunction,\,f(x)=1\text{ for }x\in K\}.
	\end{align}
	Suppose that $\vl$ is also an associative real algebra in which the product of two positive elements is positive, and in which 	$0\leq a_\lambda\uparrow b$ implies that $a_\lambda b\uparrow ab$ and $ba_\lambda \uparrow ba$ for $b\in\posvl$. If $\posmap$ is also an algebra homomorphism, then $\npm$ is a spectral measure.
\end{theorem}

\subsection{$\text{C}^\ast$-algebras}\label{subsec:C-algebras}
We conclude with a reminder of two facts about \Calgebras. There is spectral permanence in this category which for unital commutative \Calgebras\ holds in an even wider context.

\begin{lemma}\label{res:spectral_permanence}
Let $A$ be a unital complex Banach subalgebra of a unital complex Banach algebra $B$. Suppose that exist a compact Hausdorff space $\ts$ and a surjective isometric algebra homomorphism $\varphi\colon A\to\conttsC$. Then $\sigma_B(a)=\varphi(a)(\ts)$.
\end{lemma}

\begin{proof}
We need to show that $a-\lambda 1$ is not invertible in $B$ when $\lambda\in \varphi(a)(\ts)$. Suppose, to the contrary, that there exists a $b\in B$ such that $(a-\lambda 1)b=1$. Arguing locally at a point of $\ts$ where $\varphi(a)$ has the value $\lambda$, one finds, for every $\varepsilon>0$, an $a_\varepsilon\in A$ such that $\norm{a_\varepsilon}=1$ and $\norm{a_\varepsilon(a-\lambda 1}<\varepsilon$. We then have
\begin{equation*}
	1=\norm{a_\varepsilon}=\norm{a_\varepsilon(a-\lambda 1)b}\leq \norm{a_\varepsilon(f-\lambda 1}\norm{b}\leq \varepsilon \norm{b},
\end{equation*}
for all $\varepsilon>0$, which is impossible.
\end{proof}

For an element $a$ of a commutative \Calgebra, its modulus $\abs{a}$ is defined as the unique positive square root of $a^\ast a$. A $\sp\ast$-homomorphism $\rho \colon A\to B$ between two commutative \Calgebras\ preserves the modulus. Indeed, for $a\in A$, we have $\rho(\abs{a})^2=\rho({\abs{a}^2})=\rho(a^\ast a)=\rho(a)^\ast\rho(a)=\abs{\rho(a)}^2$. As positive square roots are unique, it follows that $\rho(\abs{a})=\abs{\rho(a)}$.

\section{Central operators}\label{sec:central_operators}

\noindent Let $\bl$ be a Dedekind complete real Banach lattice. The \emph{centre of $\centreblC$ of $\blC$} is defined as
\[
\centreblC\coloneqq\set{\op\in\regular(\blC): \abs{\op}\leq\lambda\idmap \text{ for some }\lambda\geq 0}.
\]
With
\[
\centrebl\coloneqq\set{\op\in\regular(\bl): \abs{\op}\leq\lambda\idmap \text{ for some }\lambda\geq 0}
\]
denoting the centre of $\bl$, it is clear that $\centreblC=\centrebl+\iu\centrebl$ as a complex linear subspace of $\regularblC$. Because $\centrebl$ is an order ideal of $\bl$, the norm on $\centreblC$ as a complex linear subspace of $\blC$ and that as a complex Banach lattice with underlying real Banach lattice $\centrebl$ coincide. Since $\centrebl$ is an algebra, so is $\centreblC$.
The order unit norm $\ounorm{\noarg}$ on $\centreblC$ is defined as
\[
\ounorm{\op}\coloneqq\inf\set{\lambda\geq 0: \abs{\op}\leq\lambda\idmap}
\]
for $\op\in\centreblC$. This is an algebra norm by \cite{huijsmans:1985}; its restriction to $\centrebl$ is the order unit norm of $\centrebl$. It is routine to verify that $\centreblC$ is a complex Banach algebra in the order unit norm.

It is well known that $\centrebl$ is a real vector lattice; see \cite[Theorem~2.43]{aliprantis_burkinshaw_POSITIVE_OPERATORS_SPRINGER_REPRINT:2006}, for example. It is likewise well known that it is commutative, but we shall also see this below.

The proof of the following theorem follows ideas from \cite{martignon:1980}.

\begin{theorem}\label{res:isomorphism}
	Let $\bl$ be a Dedekind complete real Banach lattice.
	When supplied with the conjugation $T\mapsto\overline{\op}$ and the order unit norm $\ounorm{\noarg}$, $\centreblC$ is a unital commutative \Calgebra.
\end{theorem}

\begin{proof}
	Since $\centreblC$ is a complex unital AM-space with the identity operator as order unit, the complex version of the Kakutani-Bohnenblust-Krein representation theorem (see \cite[Theorem~3.20]{abramovich_aliprantis_INVITATION_TO_OPERATOR_THEORY:2002}) yields the existence of a compact Hausdorff space $\ts$ and a surjective isometric complex linear map $\op\mapsto\widehat\op$ from $\centreblC$ onto $\conttsC$, such that $\widehat{\idmap}=\onefunction$ and  $\widehat{\overline\op}=\overline{\widehat\op}$ for $\op\in\centreblC$,  and such that its restriction to $\centrebl$  is a vector lattice isomorphism between $\centrebl$ and $\conttsR$.
	Via its restriction to $\centrebl$, this map introduces a multiplication on $\conttsR$ in which the product of two positive functions is positive, and where $\onefunction$ is the multiplicative identity element. By \cite[Proposition~1.4]{martignon:1980}, this multiplication is the pointwise one. Hence $\op\mapsto\widehat\op$ is an algebra homomorphism, implying that $\centreblC$ is commutative. For $\op\in\centrebl$, we then have
	\[
	\ounorm{\op\overline\op}=\norm{(\op\overline\op)^\widehatss}=\norm{\widehat\op\overline{\widehat\op}}=\norm{\abs{\,\widehat\op}^2}.
	\]
	Because $\norm{f^2}=\norm{f}^2$ for $f\in\conttsC$, we also have
	\[
	\ounorm{\op}^2=\norm{\widehat\op}^2=\norm{(\widehat\op)^2}=\norm{\,\abs{\widehat\op}^2}.
	\]
	Hence $\ounorm{\noarg}$ is a $\tfs{C}^\ast$\!-norm.
\end{proof}

\begin{remark}\quad
	\begin{enumerate}
	\item The representation theorem used in the above proof also asserts that $\widehat{\abs{\op}}=\abs{\widehat\op}$. Since $\,\widehat\cdot\,$ is a $\sp\ast$-homomorphism between \Calgebras, this is automatic in our case.		
	\item Via similar representation theorems, real unital AM-algebras and real AM-algebras with an approximate unit can be shown to be commutative; see \cite[Theorem~3.8]{munos-lahoz_UNPUBLISHED:2025} and \cite[Theorem~1.2]{munos-lahoz_tradacete:2025}.
\end{enumerate}
\end{remark}

We collect a few consequence of \cref{res:isomorphism}.

\begin{corollary}\label{res:norms_are_equal}
Let $\bl$ be a Dedekind complete real Banach lattice. For $\op\in\centreblC$, its order unit norm $\ounorm{\op}$, its regular norm $\norm{\op}_{\reg}$, and its norm $\norm{\op}_{\blC}$ as a bounded operator on $\blC$ are equal.	
\end{corollary}

\begin{proof}
Since
\begin{equation}\label{eq:operator_inequality}
\abs{\op z}\leq\abs{\op}\abs{z}
\end{equation}
for a general $\op\in\regularblC$ (see \cite[Theorem~92.6]{zaanen_RIESZ_SPACES_VOLUME_II:1983}), we have, for $\op\in\centreblC$,
\begin{align*}
	\norm{\op}_{\blC}&=\sup_{\norm {z}\leq 1}\norm{\op z}\\
	&=\sup_{\norm{z}\leq 1}\norm{\,\abs{\op z}\,}\\
	&\leq \sup_{\norm{z}\leq 1}\norm{\,\abs{\op}\abs{z}\,}\\
	&\leq\ounorm{\op}.
\end{align*}
As the given norm on a commutative \Calgebra\ is the minimal algebra norm on it (see \cite[Theorem~6.1]{kaplansky:1949}), we actually have equality. Because the order unit norm is a lattice norm, it also coincides with the regular norm.
\end{proof}

\begin{remark}
	For real Banach lattices, the equality of the operator norm and the order unit norm on its centre is due to Wickstead; see \cite{wickstead:1977a}.
	It is rather non-trivial that $\abs{\op z}\leq\abs{\op}\abs{z}$ for general $\op$ in the complex case, but its real analogue is elementary. Since the maximum norm is also the minimal algebra norm on $\conttsR$, the same line of reasoning shows that the operator norm and the order unit norm are equal for central operators on a real Banach lattice. This provides an alternative for the use of representations of vector lattices in $\cont^\infty(\ts)$-spaces as in \cite{wickstead:1977a}. It is not necessary that $\bl$ be Dedekind complete for this because the existence of the modulus of central operators is guaranteed by \cite[Theorem~2.0]{aliprantis_burkinshaw_POSITIVE_OPERATORS_SPRINGER_REPRINT:2006}.
\end{remark}

The following result is true for every commutative \Calgebra, unital or not, because it is isomorphic to $\contotsC$ for a locally compact Hausdorff space $\ts$.

\begin{corollary}\label{res:modulus_multiplicative}
	Let $\bl$ be a Dedekind complete real Banach lattice. Then
	\[
	\abs{\optwo \op}=\abs{\optwo}\abs{T}
	\]
	for $\optwo,\op\in\centreblC$, and
	\[
	\abs{\op\overline\op}=\abs{\op}^2
	\]
	for $\op\in\centrebl$.
\end{corollary}

We let $\tfs{Inv}(\centreblC$ denote the group of invertible elements of $\centreblC$. There is a continuous polar decomposition on this group.
Indeed, this is true for every unital commutative \Calgebra\ because it holds for $\conttsC$, as is easily verified.

\begin{corollary}\label{res:polar_decomposition_of_invertibles}
	Let $\bl$ be a Dedekind complete real Banach lattice. For $\op\in\tfs{Inv}(\centreblC)$, there exist unique $P(\op),U(\op)\in\centreblC$ such that $P(\op)\geq 0$, $\abs{U(\op)}=\idmap$, and $\op=P(\op) U(\op)$. Both $P(\op)$ and $U(\op)$ are in $\tfs{Inv}(\centreblC)$, and  the maps $\op\to P(T)$ and $\op\mapsto U(\op)$ are continuous group homomorphisms from $\tfs{Inv}(\centreblC)$ into itself.
\end{corollary}

The equality of the order unit norm $\ounorm{\noarg}$ and the operator norm $\norm{\noarg}_{\blC}$ on $\centreblC$ is used in the following proof of a Fuglede-Putnam-Rosenblum--type theorem, along the lines of the proof of \cite[Theorem~12.16]{rudin_FUNCTIONAL_ANALYSIS_SECOND_EDITION:1991}.

\begin{theorem}\label{res:FPR}
	Let $\bl$ be a Dedekind complete real Banach lattice. Take $S,T\in\centreblC$.  For a norm bounded, not necessarily order bounded, operator $\Xi$  on $\blC$, the following are equivalent;
	\begin{enumerate}
		\item\label{part:FPR_1} $S\,\Xi=\Xi\, T$;
		\item\label{part:FPT_2} $\overline S\,\Xi=\Xi\ \overline{T}$.
	\end{enumerate}
\end{theorem}

\begin{proof}
	It is sufficient to prove that part~\partref{part:FPR_1} implies part~\partref{part:FPT_2}. It follows from the assumption that
	\begin{equation}\label{eq:FPR}
		\Xi={\tfs e}^{-\overline\lambda S}\,\Xi\,{\tfs e}^{\overline\lambda T}
	\end{equation}
	for $\lambda\in\CC$.
	Define the entire function $f\colon\CC\to\nboundedblC$ from $\CC$ into the norm bounded operators on the complex Banach space $\blC$ by setting
	\[
	f(\lambda)\coloneqq {\tfs e}^{\lambda\overline S}\,\Xi\,{\tfs e}^{-\lambda\overline T}
	\]
	for $\lambda\in\CC$. By \cref{eq:FPR}, we have, using that $\centreblC$ is commutative,
	\[
	f(\lambda)={\tfs e}^{\lambda\overline S-\overline{\lambda\overline{S}}}\,\Xi\,{\tfs e}^{\overline\lambda T-\overline{\overline\lambda T}}.
	\]
	For $R\in\centreblC$ such that $\overline R=-R$, we have
	\[
	\overline{{\tfs e}^{R}}\cdot {\tfs e}^R=\idmap.
	\]
	As the order unit norm on $\centreblC$ is a $\tfs{C}^\ast$\!-norm, this implies that $\ounorm{{\tfs e}^{R}}=1$. Using this and the equality of the two norms on $\centreblC$, we see, with $\norm{\noarg}_{\blC}$ denoting the operator norm on $\nboundedblC$, that
	\begin{align*}
		\norm{f(\lambda)}_{\blC}&\leq \norm{{\tfs e}^{\lambda\overline S-\overline{\lambda\overline{S}}}}_{\blC}\,\norm{\Xi}_{\blC}\,\norm{{\tfs e}^{\overline\lambda T-\overline{\overline\lambda T}}}_{\blC}\\
		&=\ounorm{{\tfs e}^{\lambda\overline S-\overline{\lambda\overline{S}}}}\,\norm{\Xi}_{\blC}\,\ounorm{{\tfs e}^{\overline\lambda T-\overline{\overline\lambda T}}}\\
		&=\norm{\Xi}_{\blC}.
	\end{align*}
	Hence $f$ is constant. Its value at 0 is $\Xi$, so
	\[
	{\tfs e}^{\lambda\overline S}\,\Xi = \Xi\, {\tfs e}^{\lambda\overline T}.
	\]
	for $\lambda\in\CC$. Comparing coefficients, we see that part~\partref{part:FPT_2} holds.
\end{proof}

We continue locally by looking at principal order ideals.

\begin{theorem}\label{res:localisation}
	Let $\bl$ be a Dedekind complete real Banach lattice. Take a principal order ideal $J_u$ for $u>0$ of $\blC$ and a Kakutani-Bohnenblust-Krein representation $\,\widehat{\cdot}\,\colon J_u\to\conttsC$. For $\op\in\centre(J_u)$, define $\widehat\op\colon\conttsC\mapsto\conttsC$ by $\widehat T\,\widehat z=(\op z)^\widehatss$ for $z\in J_u$.
	\begin{enumerate}
		\item For $\op\in\centre(J_u)$, there exists a unique $\mult(\op)\in\conttsC$ such that $\widehat\op \widehat z=\mult(\op)\cdot \widehat z$ for $z\in J$.
		\item The map $\mult\colon\centre(J_u)\to \conttsC$ is a surjective $\sp\ast$-isomorphism of \Calgebras.
	\end{enumerate}	
\end{theorem}

\begin{proof}
	Since $\abs{\widehat\op\widehat z}\leq\ounorm{\op}\abs{\widehat z}$, $\widehat T$ preserves zeros. For every $x_0\in\ts$, $\widehat\op(\widehat z -\widehat z(x_0)\onefunction)$ vanishes at $x_0$, implying that $\widehat\op\widehat z=\widehat\op\onefunction\cdot \widehat z$. Set $\mult(\op)\coloneqq\widehat\op\onefunction=\widehat{\op u}$. It is clear that $\mult$ is a surjective algebra isomorphism.
	We also have
	\[
		\mult(\overline{\op})=\widehat{\overline{\op}u}=\widehat{\overline{\op u}}=\overline{\widehat{\op u}}=\overline{\mult(T)}.
	\]
\end{proof}

\begin{corollary}\label{res:four_equalities}
Let $\bl$ be a Dedekind complete real Banach lattice. Then the element $\abs{\op z}$ of $\centreblC$ remains unchanged when $\op$ is replaced by $\overline\op$ or $\abs{\op}$ and $z$ is replaced by $\overline z$ or $\abs{z}$.
\end{corollary}

\begin{proof}
	Take an $u\in\posbl$ such that $z\in J_u$. The restriction of $\op$ to $J_u$ is a central operator. In the notation of \cref{res:localisation} we then have, using that the $\sp\ast$-homomorphism $\mult$ automatically preserves the modulus,
	\begin{align*}
	\abs{\op z}^\widehatss&= \abs{(\op z)^\widehatss}=\abs{\widehat \op\,\widehat z}=\abs{M(\op)\cdot \widehat z}=\abs{M(T)}\cdot\abs{\widehat z}\\
	&=M(\abs{T})\cdot\abs{\widehat z}=\widehat{\abs{\op}}\abs{z}^\widehatss=(\abs{T}\abs{z})^\widehatss.
	\end{align*}
	Hence $\abs{\op z}=\abs{\op}\abs{z}$, which is evidently invariant under the replacements.
\end{proof}

\begin{remark}
The real analogue of \cref{res:four_equalities} is valid for an order bounded operator disjointness preserving operator between two Archimedean vector lattices; see \cite[Theorem~2.40]{aliprantis_burkinshaw_POSITIVE_OPERATORS_SPRINGER_REPRINT:2006}.
\end{remark}

\begin{remark}
	With \cref{res:four_equalities} available, it is immediate that $\norm{\op}_{\blC}\leq\ounorm{\op}$ for $\op\in\centreblC$. Thus the local analysis provides an alternative to the use of \cref{eq:operator_inequality} in the proof of \cref{res:norms_are_equal}.
\end{remark}

\begin{proposition}\label{res:action_of_centers}
Let $\bl$ be a Dedekind complete real Banach lattice. Take $u>0$.
\begin{enumerate}
	\item\label{part:action_of_centers_1} The map $\op\mapsto Tu$ is an isometric vector lattice isomorphism between the complex Banach lattices $(\centre(J_u),\ounorm{\noarg})$ and $(J_u,\norm{\noarg}_u)$.
	\item\label{part:action_of_centers_2} The map $\op\mapsto Tu$ is a contractive surjective vector lattice homomorphism from the complex Banach lattice $(\centreblC,\ounorm{\noarg})$ onto $(J_u,\norm{\noarg}_u)$.
\end{enumerate}	
\end{proposition}

\begin{proof}
 We prove part~\partref{part:action_of_centers_1}. The fact that $\abs{\op u}=\abs{\op}\abs{u}=\abs{\op}u$ for $\op\in\centre(J_u)$ shows that the map is a vector lattice homomorphism. It is surjective by \cref{res:localisation}. Suppose that $\op u$=0. For $z\in J_u$ with $\abs{z}\leq\lambda u$ for some $\lambda\geq 0$, we have $\abs{\op z}=\abs{\op}\abs{z}\leq\lambda\abs{\op}u=\lambda\abs{\op u}=0$. Hence it is injective.

 Take $\op\in\centreblC$. Using  in the final equality that the $\sp\ast$-isomorphism $\mult$ in \cref{res:localisation} is automatically isometric, we see, in the notation of \cref{res:localisation}, that
 \[
 \norm{Tu}_u=\norm{\widehat{\op u}}=\norm{\widehat\op\widehat u}=\norm{\mult(\op)\widehat u}=\norm{\mult(\op)\onefunction}=\norm{\mult(\op)}=\ounorm{\op}.
 \]
 Part~\partref{part:action_of_centers_2} follows from part~\partref{part:action_of_centers_1} and the fact that every central operator on $J_u$ can be extended to a central operator on $\blC$. This is a consequence of \cite[Theorem~1.26]{aliprantis_burkinshaw_POSITIVE_OPERATORS_SPRINGER_REPRINT:2006} (see also \cite[Theorem~2.49]{aliprantis_burkinshaw_POSITIVE_OPERATORS_SPRINGER_REPRINT:2006}).
\end{proof}

\section{Spectral theory}\label{sec:spectral_theory}

\noindent We start our development of spectral theory for central operators on a complex Banach lattice by collecting a few consequences of the results in \cref{sec:central_operators}. After that, we use the representation theorem in \cref{sec:preliminaries} to find spectral measures. The resulting functional calculus for an individual central operator is then a means to obtain further information.

Let $\bl$ be a Dedekind complete real Banach lattice. Take $\op\in\centreblC$. It follows from \cref{res:spectral_permanence} and \cref{res:norms_are_equal} that the spectrum of $\op$ as an element of $\centreblC$ is the same as its spectrum as an element of the regular operators or of the bounded operators on $\blC$. We shall, therefore, simply write $\specop$ for this common spectrum, which equals $\widehat\op (\ts)$ in the notation of the proof of \cref{res:isomorphism}. We shall also write $\norm{\op}$ for the coinciding order unit, regular, and operator norms of $\op$.

The following is immediate from \cref{res:isomorphism}.

\begin{proposition}
	Let $\bl$ be a Dedekind complete real Banach lattice, and let $\op\in\centreblC$. Then:
	\begin{enumerate}
		\item the spectral radius of $\op$ equals its norm;
		\item $\op\in\centrebl$ if and only if $\specop\subset\RR$;
		\item $\op\in\centrebl^+$ if and only if $\specop\subset\RR^+$;
		\item $\abs{\op}=\idmap$ if and only if $\specop\subseteq\TT$.
	\end{enumerate}
\end{proposition}

Since its kernel is a band, a central operator is determined by its restrictions to the principal ideals in a collection of such ideals whose generators generate a band in $\bl$. The same is true for its spectrum. The fact that an injective $\sp\ast$-homomorphism between \Calgebras\ is automatically isometric is the key to our understanding of this. We need the following preparation.

\begin{lemma}\label{res:closure_of_union}
	Let $\{A_\gamma:\gamma\in \Gamma\}$ be a non-empty collection of unital commutative \Calgebras. Set $A\coloneqq\prod_{\gamma\in\Gamma}A_\gamma$ and take  $x=(x_\gamma)_{\gamma\in\Gamma}$ in $A$. Then
	\begin{equation}\label{eq:spectrum_is_closure_of_union}
	\sigma_A(x)=\overline{\bigcup_{\gamma\in\Gamma}\sigma_{A_\gamma}(x_\gamma)}.
	\end{equation}
\end{lemma}

\begin{proof}
	Suppose that $\lambda\in\CC$ and that $\lambda\notin\overline{\bigcup_{\gamma\in\Gamma}\sigma_{A_\gamma}(x_\gamma)}$. Then there exists an $\varepsilon>0$ such that $\abs{\lambda-\mu}>\varepsilon$ for $\gamma\in\Gamma$ and $\mu\in\sigma_{A_\gamma}(x_\gamma)$. Since each $A_\gamma$ is isometrically isomorphic to $\cont(X_\gamma,\CC)$ for some compact Hausdorff space $X_\gamma$, this implies that $\norm{(x_\gamma-\lambda\idmap_\gamma)^{-1}}<\varepsilon^{-1}$ for all $\gamma\in\Gamma$. Hence $((x_\gamma-\lambda\idmap_\gamma)^{-1})_{\gamma\in\Gamma}$ is an element of $A$, showing that $\lambda\notin\sigma_A(x)$. \Cref{eq:spectrum_is_closure_of_union} now follows easily.
\end{proof}

\begin{theorem}
	Let $\bl$ be a Dedekind complete real Banach lattice. Take a subset $S\subseteq\posbl$ such that the band generated by $S$ in $\bl$ equals $\bl$. For each $u\in S$, let $\sigma(\op\!\!\restriction_{J_u})$ denote the spectrum of the restriction of $\op$ to $J_u$. Then
	\[
	\specop=\overline{\bigcup_{u\in S}\sigma(\op\!\!\restriction_{J_u})}.
	\]
\end{theorem}

\begin{proof}
Defined $\Psi\colon\centreblC\to\prod_{u\in S}\centre(J_u)$ by setting $(\Psi(\op))_u=\op \!\!\restriction_{J_u}$. Then $\Psi$ is an injective $\sp\ast$-homomorphism, so it is isometric. Hence $\specop$ is equal to the spectrum of $\Psi(\op)$ in $\Psi(\centreblC)$. By the general spectral permanence for \Calgebras\ (or by \cref{res:spectral_permanence}), the latter spectrum is equal to that of $\Psi(\op)$ in $\prod_{u\in S}\centre(J_u)$. An appeal to \cref{res:closure_of_union} concludes the proof.
\end{proof}

\subsection{Spectral measures and functional calculus}\label{subsec:spectral_measures_and_functional_calculus}

The idea to obtain a spectral measure is quite simple. By \cref{res:isomorphism} (and already by its proof), there exist a compact Hausdorff space $\ts$ and a surjective $\sp\ast$-isomorphism\, $\widehat\cdot\colon\centreblC\to\conttsC$. It restricts to an isomorphism $\widehat\cdot\colon\centrebl\to\conttsR$. \cref{res:riesz_representation_theorems_merged} is applicable to the inverse of this restriction, provided that $\centrebl$ is normal.

\begin{theorem}\label{res:general_spectral_measure}
	Let $\bl$ be a Dedekind complete real Banach lattice with a normal centre. Take a compact Hausdorff space $\ts$ and a surjective $\sp\ast$-isomorphism $\widehat\cdot\colon\centreblC\to\conttsC$. Then there exists a unique regular Borel measure $\npm\colon\borel(\ts)\to\centrebl$ that
	\begin{equation}\label{eq:inverting_the_isomorphism}
	\op=\ointts{\widehat T}
	\end{equation}
	for $\op\in\centrebl$. It is a spectral measure, $\npm(\mss)$ is an order projection for $\mss\in\borel(\ts)$, $\npm(\ts)=\idmap$, and \cref{eq:inverting_the_isomorphism} holds for $\op\in\centreblC$.
	
	When $V$ is a non-empty open subset of $\ts$,
	\begin{align}\label{eq:measure_of_open_subset_transform}
		\npm(V)&=\sup\,\{\op : f\in\centrebl,\, \zerofunction\leq \widehat T\leq\onefunction,\, \supp \,\widehat T\subseteq V\}.\\
	\intertext{When $K$ is a non-empty compact subset of $\ts$,}
	\label{eq:measure_of_compact_subset_transform}
		\npm(K)&=\inf\,\{\op : f\in\centrebl,\, \zerofunction\leq \widehat \op\leq\onefunction,\,\widehat\op(x)=1\text{ for }x\in K\}.
	\end{align}
	Furthermore, $\npm(V)>0$ for every non-empty open subset $V$ of $\ts$.
\end{theorem}

\begin{proof}
	The existence and uniqueness follow from  \cref{res:riesz_representation_theorems_merged}. The remaining properties are then clear; the fact that $\npm(V)>0$ when $V$ is non-empty and open is a consequence of \cref{eq:measure_of_open_subset_transform}.
\end{proof}

\begin{remark}
	The assumption on the order continuous dual of $\centrebl$ can equivalently be formulated as $\ts$ being hyper-Stonean; it is certainly Stonean because $\centrebl$ is Dedekind complete. It is satisfied when the order continuous dual of $\bl$ separates the points of $\bl$: then the maps $T\to x^\ast(\op x)$ for $x\in\bl$ and $x^\ast$ in the order continuous dual of $\vl$ provide a separating family of order continuous functionals on $\centrebl$. In particular, this assumption is satisfied when $\bl$ is order continuous.
\end{remark}

The order projections $\npm(\mss)$ originating from a unital spectral measure $\npm$ on any \sigalgebra\ with values in $\centrebl$ have properties analogous to those of the orthogonal projections in the image of a spectral measure in the Hilbert space context.
When $\mss_1\cap\mss_2=\emptyset$, $\npm(\mss_1)\npm(\mss_2)=0$, so the image bands $\npm(\mss_1)(\blC)$ and $\npm(\mss_2)(\blC)$ are disjoint. When $\mss=\bigcup_{n=1}^\infty \mss_n$ is a disjoint union, the fact that $\npm_\op(\mss)=\sosum_{n=1}^\infty \npm(\mss_n)$ is equivalent to the fact that $\npm_\op(\mss)z=\sosum_{n=1}^\infty \npm(\mss_n)z$ for $z\in\blC$. Thus the unconditional convergence of a series of pairwise orthogonal projections in the strong operator topology in the Hilbert space context is replaced with what could be called the unconditional strong $\sigma$-order convergence of a series of an analogous type in our context. When $\bl$ is order continuous, the series $\sum_{n=1}^\infty \npm(\mss_n)z$ is also unconditionally norm convergent with the same sum.

The spectral measure in \cref{res:general_spectral_measure} is for all central operators on $\blC$ simultaneously. We shall now fix $\op\in\centreblC$ and find a spectral measure for $\op$ that is defined on the Borel \sigalgebra\ of its spectrum.  Take $\ts$, $\widehat\cdot$, and $\npm$ as in \cref{res:general_spectral_measure}. We define $\npm_\op\colon\borel(\specop)\to\centrebl$ by setting
\[
\npm_\op(\mss)\coloneqq\npm\left[(\widehat\op)^{-1}(\mss)\right]
\]
for $\mss\in\borelspecop$. Clearly, $\npm_\op$ is a spectral measure, $\npm_\op(\specop)=\idmap$, and $\npm_\op(V)>0$ when $V\subseteq\specop$ is non-empty and (relatively) open.

The order integral has the same properties with respect to image measures as the Lebesgue integral; see \cite[Proposition~2.25]{de_jeu_jiang:2021c} for the precise statement. In particular,
\begin{equation*}\label{eq:image_integral}
\ointspectrum{f}=\ointts{f\circ\widehat\op}		
\end{equation*}
when $f$ is an element of the bounded Borel measurable functions $\boundedmeasfunspectrumC$ on $\specop$. When we let $\idfunction\colon\sigma(K)\to\sigma(K)$ denote the identity function, we see from \cref{eq:inverting_the_isomorphism} that
\[
\op=\ointspectrum{\idfunction}.	
\]
We introduce a functional calculus $\rho_\op\colon \boundedmeasfunspectrumC \to\centreblC$ by defining
\begin{equation}\label{eq:functional_calculus_definition}
\rho_\op(f)\coloneqq\ointspectrum{f}
\end{equation}
for $f\in\boundedmeasfunspectrumC$.\footnote{The natural domain of $\rho_\op$ consists of the $\npm_\op$-integrable functions on $\specop$. It follows from \cite[Lemma~2.29]{de_jeu_jiang:2021c}, however, that every $\npm_\op$-integrable function is $\npm_\op$-almost everywhere equal to a bounded measurable function. Hence one may just as well restrict the definition to the bounded case.} Since $\npm_\op$ is a spectral measure, \cite[Theorem~4.6]{de_jeu_jiang:2025} shows that $\rho_\op$ is a unital $\sp\ast$-homomorphism. This functional calculus is our means to obtain further information about $\op$. It is also used in the proof of the following result.

\begin{theorem}\label{res:specific_spectral_measure_uniqueness}
	Let $\bl$ be a Dedekind complete real Banach lattice with a normal centre. Take $\op\in\centreblC$. There exists a unique spectral Borel measure $\npm_\op\colon\borel(\specop)\to\centrebl$ with $\npm_\op(\specop)=\idmap$ such that
		\begin{equation}\label{eq:specific_spectral_measure_uniqueness}
		T=\ointspectrum{\idfunction}.
		\end{equation}
		It is a regular Borel measure, $\npm_\op(\mss)$ is an order projection for $\mss\in\borel(\sigma{\op})$, and $\npm_op(\specop)=\idmap$. Furthermore, $\npm_\op>0$ for every non-empty \uppars{relatively} open subset of $\specop$.
\end{theorem}

\begin{proof}
	The existence of a Borel measure $\mu_\op$ with the asserted properties has already been established. It remains to show that it is unique as a unital spectral Borel measure such that \cref{eq:specific_spectral_measure_uniqueness} holds. For this, we first note that every Borel measure $\nu\colon\borel(\specop)\to\centrebl$ is regular. Indeed, for every positive order continuous functional $x^\ast$ on $\centrebl$, $x^\ast\circ\nu$ is a positive real-valued Borel measure on $\specop$. It is regular by \cite[Theorem~7.8]{folland_A_COURSE_IN_ABSTRACT_HARMONIC_ANALYSIS_SECOND_EDITION:2016}, and then $\nu$ itself is regular by \cite[Corollary~5.3]{de_jeu_jiang:2022b}.
Suppose that $\nu\colon\borel(\specop)\to\centrebl$ is a spectral Borel measure with $\nu(\specop)=\idmap$ such that
\[
\op=\ointspectrumalt{\idfunction}.
\]
Define $\rho_\op^\prime\colon\contspectrumC \to\centrebl$ by setting
\[
\rho_\op^\prime(f)\coloneqq\ointspectrumalt{f}
\]
for $f\in\contspectrumC$. Again by \cite[Theorem~4.6]{de_jeu_jiang:2025},  $\rho_\op^\prime$ is also a $\sp\ast$-homomorphism. As $\rho_\op^\prime(\idfunction)=\op$ by assumption, we also have  $\rho_\op^\prime(\overline{\idfunction})=\overline{\op}$. Hence $\rho_\op$ and $\rho_\op^\prime$ agree on $\idmap$, $\idfunction$, and $\overline{\idfunction}$.
Using the Stone-Weierstrass theorem, it follows that they agree on $\contspectrumC$, and in particular on $\contspectrumR$. As we know $\nu$ to be regular, $\nu=\npm_\op$ by the uniqueness statement in \cref{res:riesz_representation_theorems_merged}.
\end{proof}

We shall refer to $\npm_\op$ in \cref{res:specific_spectral_measure_uniqueness} as the spectral measure of $\op$, and to $\rho_\op$ from  \cref{eq:functional_calculus_definition} as its functional calculus. We continue by investigating $\rho_\op$.

Being a $\sp\ast$-homomorphism, $\rho_\op$ preserves the modulus. Therefore, if $f\in\boundedmeasfunspectrumC$, then $\rho_\op(f)=0$ if and only if $f$ vanishes $\npm_\op$-almost everywhere.

When $f\in\contspectrumC$, we have
\[
\rho_\op(f)^\widehatss=\left(\ointspectrum{f}\right)^\widehatss=\left(\ointts{f\circ\widehat\op}\right)^\widehatss=f\circ\widehat\op.
\]
This implies that
\[
\sigma(\rho_\op(f))=\rho_\op(f)^\widehatss(\ts)=(f\circ\widehat\op)(\ts)=f(\specop),
\]
and also that
\[
\norm{\rho_\op(f)}=\norm{\rho_\op(f)^\widehatss}=\norm{f\circ\widehat\op}=\norm{f}
\]
for $f\in\contspectrumC$. Since $\rho_\op(\overline \idfunction)=\overline\op$, we then see, using the Stone-Weierstrass theorem, that the image of $\contspectrumC$ under $\rho_\op$ is the closed unital subalgebra of $\centreblC$ that is generated by $\idmap$, $\op$, and $\overline\op$.

We have now established the first four parts of the following result. The first convergence in its fifth part follows from the dominated convergence theorem (see \cite[Theorem~6.13]{de_jeu_jiang:2022a}); it implies the second.\footnote{Since $(\rho_\op(f_n))_{n=1}^\infty$ is an order bounded sequence in the orthomorphisms on $\blC$, these two convergences are actually equivalent; see \cite[Theorem~9.4]{deng_de_jeu:2021}.}

\begin{theorem}\label{res:specific_spectral_measure}
Let $\bl$ be a Dedekind complete real Banach lattice with a normal centre. Take $\op\in\centreblC$.
\begin{enumerate}
	\item\label{part:specific_spectral_measure_1} The functional calculus
	\[
	\rho_\op\colon\boundedmeasfunspectrumC\to\centreblC
	\]
	of $\op$ is a unital $\sp\ast$-homomorphism such that $\rho_\op(\idfunction)=\op$.
	\item $\rho_\op$ establishes an isomorphism between $\contspectrumC$ and the closed unital subalgebra of $\centreblC$ that is generated by $\op$ and $\overline{\op}$.
	\item\label{part:specific_spectral_measure_3}  We have $\sigma(\rho_\op(f))=f(\specop)$ for $f\in\contspectrumC$.
	\item For $f\in\boundedmeasfunspectrumC$, $\rho_\op(f)=0$ if and only if $f=0$ $\npm_\op$-almost everywhere.
	\item\label{part:specific_spectral_measure_4}  \label{part:convergences} Suppose that $(f_n)_{n=1}^\infty\subseteq\boundedmeasfunspectrumC$ is a $\npm_\op$-almost everywhere uniformly bounded sequence which is $\npm_\op$-almost everywhere convergent to $f\in\boundedmeasfunspectrumC$. Then $\rho_\op(f_n)\soto\rho_\op(f)$ in $\centreblC$ and $\rho_\op(f_n)z\soto\rho_\op(f)z$ in $\blC$ for $z\in\blC$.
\end{enumerate}
\end{theorem}	

The convergence $\rho_\op(f_n)\soto\rho_\op(f)$ in part~\partref{part:convergences} is also in $\regularblC$, of which $\centreblC$ is an order ideal.
When $\bl$ is order continuous, the convergence $\rho_\op(f_n)z\soto\rho_\op(f)z$ in part~\partref{part:convergences} is also in norm.

\begin{remark}\label{rem:freudenthal}
It follows from part~\partref{part:specific_spectral_measure_1} of \cref{res:specific_spectral_measure} that $\op$ is a uniform limit of linear combinations $\sum_{i=1}^n\lambda_i\proj{i}$ of disjoint order projections where the $\lambda_i$ are in $\specop$.  It follows already from Freudenthal's theorem that $\op$ is a uniform limit of such linear combinations, but not that they can be chosen to be in $\specop$.
\end{remark}	

More can be said about the image of $\boundedmeasfunspectrumR$ and of its positive cone under $\rho_\op$ when $\centrebl$ has the countable sup property. This is, e.g., the case when $\bl$ is separable; see \cite[Section~2.2]{de_jeu_jiang:2021c}. For this, we need a definition.

\begin{definition}\label{def:up-down}
Let $S$ be a non-empty subset of $\centrebl$. We define
\begin{align*}
	S^\up&\coloneqq\big\{\!\tpskip x\in\centrebl\tnskip:\! \!\text{ there exists a net }\net{x}\! \text{ in }\!S\!\text{ such that } x_\lambda\!\uparrow\! x \text{ in }\centrebl\big\}\\
   \intertext{and its sequential version}
	S^\ups&\coloneqq\left\{\!\tpskip x\in\centrebl\tnskip:\!\! \text{ there exists a sequence }\seq{x}\! \text{ in }\!S\!\text{ such that } x_n\!\uparrow\! x \text{ in }\centrebl\!\right\}\!,
\end{align*}
and define $S^\down$ and $S^\downs$ similarly.\footnote{In \cite{aliprantis_burkinshaw_POSITIVE_OPERATORS_SPRINGER_REPRINT:2006}, our $S^\up$, $S^\down$, $S^\ups$, and $S^\downs$ are denoted by $S^\uparrow$, $S^\downarrow$, $S^\upharpoonleft$, and $S^\downharpoonleft$, respectively; in \cite{de_pagter:1983}, they are $S^\uparrow$, $S^\downarrow$, $S^{\uparrow_\omega}$, and $S^{\downarrow_\omega}$, respectively. Our notation may be a little clearer in a smaller font.} We set $S^{\ups\downs}\coloneqq\left(S^\ups\right)^{\raisebox{-2.5pt}{$\scriptstyle\downs$}}$ and $S^{\downs\ups}\coloneqq\left(S^\downs\right)^{\raisebox{-2.5pt}{$\scriptstyle\ups$}}$.
\end{definition}

The next result, the hypotheses of which are satisfied for all separable order continuous real Banach lattices, is the specialisation of \cite[Theorem~6.4(8)]{de_jeu_jiang:2021c} to our case. The key to its proof is the use of the monotone convergence theorem for the order integral (see \cite[Theorem~6.9]{de_jeu_jiang:2022b}) to exploit the countable sup property. Recall that the image of $\contspectrumR$ under $\rho_\op$ is the closed unital subalgebra of $\centrebl$ that is generated by $\op$.

\begin{theorem}\label{eq:ups_and_downs}
	Let $\bl$ be a \Dc\ real Banach lattice with a normal centre. Suppose that $\centrebl$ has the countable sup property. Take $\op\in\centrebl$. Then:
	\begin{align*}
		\rho_\op(\boundedmeasfunspectrumR^+)&=[\rho_\op(\boundedmeasfunspectrumR^+)]^\up=[\rho_\op(\boundedmeasfunspectrumR^+)]^\down	\\
		&=[\rho_\op(\contspectrumR^+)]^{\ups\downs}=[\rho_\op(\contspectrumR^+)]^{\downs\ups}
		\intertext{and}
		\rho_\op(\boundedmeasfunspectrumR)&=[\rho_\op(\boundedmeasfunspectrumR)]^\up=[\rho_\op(\boundedmeasfunspectrumR)]^\down	\\
		&=[\rho_\op(\contspectrumR)]^{\ups\downs}=[\rho_\op(\contspectrumR)]^{\downs\ups}.
	\end{align*}
\end{theorem}

We turn to commutants. The following follows easily from \cref{res:FPR}, the fact that the order unit norm and the operator norm coincide on $\centreblC$, and the Stone-Weierstrass theorem.

\begin{proposition}\label{res:commutant_for_bounded_operators}
	Let $\bl$ be a Dedekind complete real Banach lattice with a normal centre. Take $\op\in\centreblC$. Then the following are equivalent for a norm bounded, not necessarily regular, operator $\Xi$ on $\blC$:
	\begin{enumerate}
		\item $\Xi$ commutes with $\op$;
		\item $\Xi$ commutes with $\overline\op$;
		\item $\Xi$ commutes with $\rho_\op(f)$ for $f\in\contspectrumC$.
	\end{enumerate}	
\end{proposition}

When $\Xi$ is order continuous, the spectral projections enter the picture.

\begin{proposition}\label{res:commutant_for_order_continuous_operators}
	Let $\bl$ be a Dedekind complete real Banach lattice with a normal centre. Take $\op\in\centreblC$. Then the following are equivalent for an order continuous operator $\Xi$ on $\blC$:
	\begin{enumerate}
		\item\label{part:commutant_for_order_continuous_operators_1} $\Xi$ commutes with $\op$;
		\item\label{part:commutant_for_order_continuous_operators_2} $\Xi$ commutes with $\overline\op$;
		\item\label{part:commutant_for_order_continuous_operators_3} $\Xi$ commutes with $\rho_\op(f)$ for $f\in\contspectrumC$;
		\item\label{part:commutant_for_order_continuous_operators_4} $\Xi$ commutes with $\npm_\op(\mss)$ for $\mss\in\borelspecop$;
		\item\label{part:commutant_for_order_continuous_operators_5} $\Xi$ commutes with $\rho_\op(f)$ for $f\in\boundedmeasfunspectrumC$.
	\end{enumerate}	
\end{proposition}

\begin{proof}
	
We recall that an order continuous operator on a real vector lattice is order bounded. Hence $\Xi$, which is then also regular, is norm bounded.

We know from \cref{res:commutant_for_bounded_operators} that the parts~\partref{part:commutant_for_order_continuous_operators_1}, \partref{part:commutant_for_order_continuous_operators_2}, and~\partref{part:commutant_for_order_continuous_operators_3} are equivalent.

When $\Xi$ commutes with $\rho_\op(f)$ for $f\in\contspectrumC$, its order continuity and \cref{eq:measure_of_open_subset} show that it commutes with $\npm_T(V)$ when $V\subseteq\specop$ is (relatively) open. The order continuity and the outer regularity of $\npm_\op$ then imply that it commutes with $\npm_\op(\mss)$ for $\mss\in\borelspecop$.  Hence part~\partref{part:commutant_for_order_continuous_operators_3} implies part~\partref{part:commutant_for_order_continuous_operators_4}.

It follows from the fact that every $f\in\boundedmeasfunspectrumC$ can be approximated uniformly by step functions and the norm continuity of $\rho_\op$ that part~\partref{part:commutant_for_order_continuous_operators_4} implies part~\partref{part:commutant_for_order_continuous_operators_5}.

It is trivial that part~\partref{part:commutant_for_order_continuous_operators_5} implies part~\partref{part:commutant_for_order_continuous_operators_1}.
\end{proof}

It follows from \cref{res:commutant_for_order_continuous_operators} that, similarly to normal operators, the spectral projections of $\op\in\centreblC$ are in its bicommutant in the following sense.

\begin{corollary}\label{res:bicommutant}
Let $\bl$ be a Dedekind complete real Banach lattice with a normal centre. Take $\op\in\centreblC$. For $f\in\boundedmeasfunspectrumC$, $\rho_\op(f)$ commutes with every order continuous operator on $\blC$ that commutes with $\op$. In particular, $\npm_\op(\mss)$ has this property for $\mss\in\borelspecop$.
\end{corollary}

As a first application of the functional calculus, we include a polar decomposition for general central operators, supplementing \cref{res:polar_decomposition_of_invertibles}. The easy proof of \cite[Theorem~12.35(b)]{rudin_FUNCTIONAL_ANALYSIS_SECOND_EDITION:1991} yields the following in our case.

\begin{proposition}\label{res:polar_decomposition_general}
	Let $\bl$ be a Dedekind complete real Banach lattice with a normal centre. Take $\op\in\centreblC$. Then there exist $P,U\in\rho_\op(\boundedmeasfunspectrumC)\subseteq\centreblC$ such that $P\geq 0$, $\abs{U}=\idmap$, and $\op=PU$.
\end{proposition}	

\subsection{Eigenvalues}\label{subsec:eigenvalues}
We start with a few remarks.
For $\op\in\centreblC$ and $\lambda\in\CC$, the kernel of $\op-\lambda\idmap$ is a band. Hence it is generated by its positive elements. For $\mss\in\borelspecop$, $\npm_\op(\mss)z=0$ for $z=x+ \iu y$ if and only if $\npm_\op(\mss)x^\pm=\npm_\op(\mss)y^\pm=0$. With this in mind, the following is easily established.

\begin{lemma}\label{res:vanishing}
	Let $\mss_1,\mss_2,\ldots\in\borel(\specop)$ be pairwise disjoint. Take $z\in\blC$. Then $\npm_\op(\bigcup_{n=1}^\infty\mss_n)z=0$ if and only if $\npm_\op(\mss_n)z=0$ for $n\geq 1$.
\end{lemma}

With \cref{res:vanishing} available, the methods of proof of \cite[Theorem~12.28]{rudin_FUNCTIONAL_ANALYSIS_SECOND_EDITION:1991} yields the following.

\begin{proposition}
Let $\bl$ be a Dedekind complete real Banach lattice with a normal centre. Take $\op\in\centreblC$. Then
\[
\ker \rho_\op(f)=\npm_\op(\set{\lambda\in\specop: f(\lambda)=0})(\blC)
\]
for $f\in\boundedmeasfunspectrumC$.
\end{proposition}

For $\lambda\in\specop$, set $\proj{\lambda}\coloneqq\npm_\op(\{\lambda\})$. On observing that
\begin{equation}\label{eq:zero_product}
\proj{\lambda}\proj{\mu}=\npm_\op(\{\lambda\}\cap\{\mu\})=\npm_\op(\emptyset)=0
\end{equation}
when $\lambda,\mu\in\specop$ are different, the following is now clear. Its part~\partref{part:disjointness_of_eigenbands} also holds when $\lambda,\mu\in\RR$ are different and $\op$ is an orthomorphism on a real vector lattice; see \cite[Lemma~2.1]{boulabiar_buskes_sirotkin:2006}.

\begin{corollary}\label{res:eigenvalues}
	Let $\bl$ be a Dedekind complete real Banach lattice with a normal centre.  Take $\op\in\centreblC$.  Then:
	\begin{enumerate}
		\item $\ker (\op-\lambda\idmap)=\proj{\lambda}(\blC)$ for $\lambda\in\specop$;
		\item $\lambda\in\specop$ is an eigenvalue of $\op$ if and only if $\proj{\lambda}\neq 0$;
		\item every isolated point of $\specop$ is an eigenvalue of $\op$;	
		\item\label{part:disjointness_of_eigenbands} $\ker(\op-\lambda\idmap)\perp\ker(\op-\mu\idmap)$ when $\lambda,\mu\in\CC$ are different.		
\end{enumerate}
\end{corollary}

The following is the ordered analogue of \cite[Theorem~12.29]{rudin_FUNCTIONAL_ANALYSIS_SECOND_EDITION:1991}.  It applies, in particular, to compact central operators with infinite spectrum. Finite spectra give a similar result for general central operators.

\begin{theorem}\label{res:eigenvalue_expansion}
	Let $\bl$ be a Dedekind complete real Banach lattice with a normal centre.  Take $\op\in\centreblC$, and suppose that $\specop$ is a countable infinite set $\set{\lambda_1,\lambda_2,\dotsc}$. Then:
	\begin{enumerate}
		\item\label{part:eigenvalue_expansion_1}  $\idmap=\sosum_{i=1}^\infty \proj{\lambda_i}$ and $\op=\sosum_{i=1}^\infty\lambda_i\proj{\lambda_i}$ in $\centreblC$;
		\item\label{part:eigenvalue_expansion_2} for $z\in\blC$, there exist unique $z_i$ such that $(\op-\lambda_i\idmap)z_i=0$ and \[
		z=\sosum_{i=1}^\infty z_i
		\]
		in $\blC$. In fact, $z_i=\proj{\lambda_i}z$.
	\end{enumerate}
\end{theorem}

\begin{proof}
For $n=1,2,\dotsc$, set $f_n\coloneqq \sum_{i=1}^n \chi_{\lambda_i}$. \cref{res:specific_spectral_measure} yields that $\idmap=\sosum_{i=1}^\infty \proj{\lambda_i}$ in $\centreblC$, and that $z=\sosum_{i=1}^\infty \proj{\lambda_i}z$ in $\bl$. Clearly, $(\op-\lambda_i\idmap)\proj{\lambda_i}z=0$. If $(\op-\lambda_i\idmap)z_i=0$ and $z=\sosum_{i=1}^\infty z_i$ in $\blC$, then applying $\proj{\lambda_j}$, which is $\sigma$-order continuous, shows that $z_j=\proj{\lambda_j}z$ for all $j$.

The choice $f_n\coloneqq \sum_{i=1}^n \lambda_i\chi_{\lambda_i}$ gives the second convergence in part~\partref{part:eigenvalue_expansion_1}.
\end{proof}

When $\bl$ is order continuous, the series for $z$ in part~\partref{part:eigenvalue_expansion_2} is also unconditionally norm convergent. Since the $\proj{\lambda_j}$ are also norm continuous, the same argument shows that it is then also uniquely determined as a norm convergent series.

When $\specop$ is finite, the series in \cref{res:eigenvalue_expansion} become finite sums.  A little more can be said. Suppose that
$\op=\sum_{i=1}^n \lambda_i\proj{\lambda_i}$ with $\proj{\lambda_i}\cdot\proj{\lambda_j}=0$ when $i\neq j$. Then $p(\op)=\sum_{i=1}^n p(\lambda_i)\proj{\lambda_i}$ for every polynomial $p$. With this observation, the proof of the next result is easy.

\begin{corollary}\label{res:finite_spectrum}
	Let $\bl$ be a Dedekind complete real Banach lattice with a normal centre. Take $\op\in\centreblC$.  Then the following are equivalent:
	\begin{enumerate}
		\item\label{part:finite_spectrum_1} $\specop$ is a finite set $\set{\lambda_1,\ldots,\lambda_n}$;
		\item there exists a polynomial $p$ such that $p(\op)=0$.
	\end{enumerate}
	When this is the case, $\idmap=\sum_{i=1}^n$ and $\op=\sum_{i=1}^n \lambda_i\proj{\lambda_i}$, and every $z\in\blC$ can uniquely be written as a sum
	\[
	z=\sum_{i=1}^n z_i
	\]
	where $\op z_i=\lambda_i z_i$. In fact, $z_i=\proj{\lambda_i}z$.
\end{corollary}

In the terminology of \cite{boulabiar_buskes_sirotkin:2006}, \cref{res:finite_spectrum} states that $\op\in\centreblC$ is strongly diagonal if and only if it is algebraic. This is also true for orthomorphisms on a real vector lattice; see \cite[Theorem~3.3]{boulabiar_buskes_sirotkin:2006}.

As for \cite[Theorem~12.30]{rudin_FUNCTIONAL_ANALYSIS_SECOND_EDITION:1991}, the norm continuity of the functional calculus of a central operator gives our final result.

\begin{theorem}\label{res:compact_operators}
	Let $\bl$ be a Dedekind complete real Banach lattice with a normal centre. Take $\op\in\centreblC$. Then $\op$ is compact if and only the following conditions are both met:
	\begin{enumerate}
		\item $\specop$ has no limit point except possibly zero;
		\item $\dim\ker(\op-\lambda\idmap)<\infty$ for all non-zero $\lambda\in\CC$.
	\end{enumerate}	
\end{theorem}

\subsection*{Acknowledgements}  The results in this paper were obtained, in part, during a visit of the first author to Sichuan University. The generous support by the Erasmus+ ICM programme(KA171), National Natural Science Foundation of China (Grand No. 12201439) and Natural Science Foundation of Sichuan Province(Grand No. 2024NSFSC1339) that made this possible is gratefully acknowledged. The authors thank Miloud Chil and Hamza Hafsi for putting forward the possibility of a spectral theorem for orthomorphisms on real vector lattices.

\bibliographystyle{plain}
\urlstyle{same}

\bibliography{general_bibliography}

\def\cprime{$'$}
  \def\lfhook#1{\setbox0=\hbox{#1}{\ooalign{\hidewidth\lower1.5ex\hbox{'}\hidewidth\crcr\unhbox0}}}
\begin{thebibliography}{10}

\bibitem{abramovich_aliprantis_INVITATION_TO_OPERATOR_THEORY:2002}
Y.A. Abramovich and C.D. Aliprantis.
\newblock {\em An invitation to operator theory}, volume~50 of {\em Graduate
  Studies in Mathematics}.
\newblock American Mathematical Society, Providence, RI, 2002.

\bibitem{aliprantis_burkinshaw_PRINCIPLES_OF_REAL_ANALYSIS_THIRD_EDITION:1998}
C.D. Aliprantis and O.~Burkinshaw.
\newblock {\em Principles of real analysis}.
\newblock Academic Press, Inc., San Diego, CA, third edition, 1998.

\bibitem{aliprantis_burkinshaw_POSITIVE_OPERATORS_SPRINGER_REPRINT:2006}
C.D. Aliprantis and O.~Burkinshaw.
\newblock {\em Positive operators}.
\newblock Springer, Dordrecht, 2006.
\newblock Reprint of the 1985 original.

\bibitem{boulabiar_buskes_sirotkin:2006}
K.~Boulabiar, G.~Buskes, and G.~Sirotkin.
\newblock Algebraic order bounded disjointness preserving operators and
  strongly diagonal operators.
\newblock {\em Integral Equations Operator Theory}, 54(1):9--31, 2006.

\bibitem{de_jeu_jiang:2021c}
M.~de~Jeu and X.~Jiang.
\newblock {R}iesz representation theorems for positive algebra homomorphisms.
\newblock Preprint, 2021. Available at
  \url{https://arxiv.org/pdf/2109.10690.pdf}.

\bibitem{de_jeu_jiang:2025}
M.~de~Jeu and X.~Jiang.
\newblock {R}iesz representation theorems for vector lattices and {B}anach
  lattices of regular operators.
\newblock Preprint, 2025. Available at \url{https://arxiv.org/pdf/2508.12568}.

\bibitem{de_jeu_jiang:2022a}
M.~de~Jeu and X.~Jiang.
\newblock Order integrals.
\newblock {\em Positivity}, 26(2):Paper No. 32, 2022.

\bibitem{de_jeu_jiang:2022b}
M.~de~Jeu and X.~Jiang.
\newblock Riesz representation theorems for positive linear operators.
\newblock {\em Banach J. Math. Anal.}, 16(3):Paper No. 44, 2022.

\bibitem{de_pagter:1983}
B.~de~Pagter.
\newblock The components of a positive operator.
\newblock {\em Nederl. Akad. Wetensch. Indag. Math.}, 45(2):229--241, 1983.

\bibitem{deng_de_jeu:2021}
Y.~Deng and M.~de~Jeu.
\newblock Convergence structures and locally solid topologies on vector
  lattices of operators.
\newblock {\em Banach J. Math. Anal.}, 15(3):Paper No. 57, 2021.

\bibitem{folland_REAL_ANALYSIS_SECOND_EDITION:1999}
G.B. Folland.
\newblock {\em Real analysis. {M}odern techniques and their applications}.
\newblock Pure and Applied Mathematics. John Wiley \& Sons, Inc., New York,
  second edition, 1999.

\bibitem{folland_A_COURSE_IN_ABSTRACT_HARMONIC_ANALYSIS_SECOND_EDITION:2016}
G.B. Folland.
\newblock {\em A course in abstract harmonic analysis}.
\newblock Textbooks in Mathematics. CRC Press, Boca Raton, FL, second edition,
  2016.

\bibitem{huijsmans:1985}
C.B. Huijsmans.
\newblock An inequality in complex {R}iesz algebras.
\newblock {\em Studia Sci. Math. Hungar.}, 20:29--32, 1985.

\bibitem{kaplansky:1949}
I.~Kaplansky.
\newblock Normed algebras.
\newblock {\em Duke Math. J.}, 16:399--418, 1949.

\bibitem{kusraev_tasoev:2017}
A.G. Kusraev and B.B. Tasoev.
\newblock Kantorovich-{W}right integration and representation of vector
  lattices.
\newblock {\em J. Math. Anal. Appl.}, 455(1):554--568, 2017.

\bibitem{martignon:1980}
L.~Martignon.
\newblock Banach {$f$}\!-algebras and {B}anach lattice algebras with unit.
\newblock {\em Bol. Soc. Brasil. Mat.}, 11(1):11--17, 1980.

\bibitem{munos-lahoz_UNPUBLISHED:2025}
D.~Mu\~noz{-}Lahoz.
\newblock {$f$}\!-algebra products on {AL} and {AM}-spaces.
\newblock Preprint, 2025. Available at \url{https://arxiv.org/pdf/2507.08435}.

\bibitem{munos-lahoz_tradacete:2025}
D.~Mu\~noz{-}Lahoz and P.~Tradacete.
\newblock Banach lattice {AM}-algebras.
\newblock {\em Proc. Amer. Math. Soc.}, 153(6):2565--2577, 2025.

\bibitem{rudin_FUNCTIONAL_ANALYSIS_SECOND_EDITION:1991}
W.~Rudin.
\newblock {\em Functional analysis}.
\newblock International Series in Pure and Applied Mathematics. McGraw-Hill,
  Inc., New York, second edition, 1991.

\bibitem{schaefer_BANACH_LATTICES_AND_POSITIVE_OPERATORS:1974}
H.H. Schaefer.
\newblock {\em Banach lattices and positive operators}, volume 215 of {\em Die
  {G}rundlehren der mathematischen {W}issenschaften}.
\newblock Springer-Verlag, New York-Heidelberg, 1974.

\bibitem{wickstead:1977a}
A.W. Wickstead.
\newblock Representation and duality of multiplication operators on
  {A}rchimedean {R}iesz spaces.
\newblock {\em Compositio Math.}, 35(3):225--238, 1977.

\bibitem{wickstead:1982}
A.W. Wickstead.
\newblock Spectral properties of compact lattice homomorphisms.
\newblock {\em Proc. Amer. Math. Soc.}, 84(3):347--353, 1982.

\bibitem{wright:1969b}
J.D.M. Wright.
\newblock A {R}adon-{N}ikodym theorm for {S}tone algebra valued measures.
\newblock {\em Trans. Amer. Math. Soc.}, 139:75--94, 1969.

\bibitem{wright:1969a}
J.D.M. Wright.
\newblock Stone-algebra-valued measures and integrals.
\newblock {\em Proc. London Math. Soc. (3)}, 19:107--122, 1969.

\bibitem{zaanen_RIESZ_SPACES_VOLUME_II:1983}
A.C. Zaanen.
\newblock {\em Riesz spaces. {II}}, volume~30 of {\em North-Holland
  Mathematical Library}.
\newblock North-Holland Publishing Co., Amsterdam, 1983.

\bibitem{zaanen_INTRODUCTION_TO_OPERATOR_THEORY_IN_RIESZ_SPACES:1997}
A.C. Zaanen.
\newblock {\em Introduction to operator theory in {R}iesz spaces}.
\newblock Springer-Verlag, Berlin, 1997.

\end{thebibliography}

\end{document}